\documentclass[12pt,a4paper,reqno]{amsart}

\usepackage{amssymb}
\usepackage{amsfonts}
\usepackage{amsthm}
\usepackage[all,cmtip]{xy}

\newcommand{\R}{\mathbb{R}}
\newcommand{\Z}{\mathbb{Z}}
\newcommand{\N}{\mathbb{N}}
\newcommand{\F}{\mathbb{F}}
\newcommand{\D}{\mathcal{D}}
\newcommand{\s}{\mathcal{S}}
\newcommand{\Hom}{\mathrm{Hom}}
\newcommand{\PSL}{\mathrm{PSL}}
\newcommand{\deff}{\mathrm{def}\,}
\newcommand{\dn}[1]{\|#1\|_{\mathrm{def}}}
\newcommand{\h}{\mathrm{H}}
\newcommand{\eh}{\mathrm{EH}}
\newcommand{\bb}{\mathrm{b}}
\newcommand{\im}{\mathrm{im}\,}
\newcommand{\QM}{\mathrm{QM}}
\newcommand{\HQM}{\mathrm{HQM}}
\newcommand{\QMa}{\mathrm{QM}_{\mathrm{alt}}}
\newcommand{\QZ}{\mathrm{Q}\mathcal{Z}}
\newcommand{\QZa}{\mathrm{Q}\mathcal{Z}_{\mathrm{alt}}}
\newcommand{\lia}{\ell^\infty_{\mathrm{alt}}}
\newcommand{\QRep}{\mathrm{QRep}}
\newcommand{\QRepa}{\mathrm{QRep}_{\mathrm{alt}}}
\newcommand{\ZZ}{\mathcal{Z}}
\newcommand{\Map}{\mathrm{Map}}

\newcommand{\Aut}{\mathrm{Aut}}
\newcommand{\Out}{\mathrm{Out}}
\newcommand{\Stab}{\mathrm{Stab}}
\newcommand{\infplus}{\oplus_\infty}
\newcommand{\hooklongrightarrow}{\lhook\joinrel\longrightarrow}

\usepackage[top=3cm, bottom=3cm, left=3cm, right=3cm]{geometry} 

\theoremstyle{plain}
\newtheorem{theorem}[]{Theorem}
\newtheorem{theorem_app}[subsection]{Theorem}
\newtheorem*{conjecture}{Conjecture}

\newtheorem{proposition}[theorem]{Proposition}
\newtheorem{proposition_app}[subsection]{Proposition}
\newtheorem*{proposition*}{Proposition}
\newtheorem{lemma}[theorem]{Lemma}
\newtheorem{corollary}[theorem]{Corollary}
\newtheorem{corollary_app}[subsection]{Corollary}
\newtheorem*{corollary*}{Corollary}
\theoremstyle{definition}
\newtheorem*{example}{Example}
\newtheorem*{remark}{Remark}
\newtheorem*{question}{Question}
\newtheorem*{questions}{Questions}

\newtheoremstyle{theorem_num}
{}{}
{\itshape}
{}
{\bfseries}
{.}
{ }
{\thmname{#1}\thmnote{ \bfseries #3}}

\theoremstyle{theorem_num}
\newtheorem{thmn}{Theorem}
\newtheorem{corn}{Corollary}
\newtheorem{propn}{Proposition}

\begin{document}

\title{Split quasicocycles}
\author{Pascal Rolli}
\date{\today}
\thanks{The author was supported by the Swiss
National Science Foundation projects 127016 and 144373}
\address{Departement Mathematik, ETH Z\"urich, CH-8092 Z\"urich, Switzerland}
\email{pascal.rolli@math.ethz.ch}

\begin{abstract}
Let $E$ be a linear isometric representation of a group $\Gamma$. In this paper we construct and study a family of quasicocycles $\Gamma\longrightarrow E$ that arise from splittings $\Gamma=A\ast B$. Under certain assumptions on $A$, $B$ and $E$ the bounded cohomology classes associated to these quasicocycles form an infinite-dimensional subspace of $\h^2_\bb(\Gamma,E)$. This is in particular the case when $\Gamma$ is free and $E$ finite-dimensional or of the type $\ell^p(\Gamma)$. For the trivial target $E=\R$ we obtain a new family of quasimorphisms for which we compute the Gromov norm in bounded cohomology. This yields a linear isometric embedding $\D(A)\oplus\D(B)\hooklongrightarrow\h^2_\bb(\Gamma,\R)$, where $\D(A)$ is a Banach space which is norm-equivalent to the alternating subspace of $\ell^\infty(A)$. We prove that there are classes of our type in $\h^2_\bb(\F_2,\R)$ which have infinite stabilizer under the natural action of $\Out(\F_2)$. By replacing the target $E$ with a group $G$ with bi-invariant metric we obtain a new type of quasi-representations $\Gamma\longrightarrow G$ that arise from splittings of $\Gamma$.
\end{abstract}

\maketitle
\vspace{-.6cm}
\section{Introduction}
\subsection{Statement of the results}In the preprint \cite{R} we observed that a pair of alternating bounded maps on $\Z$ can be combined to obtain a quasimorphism on the free group on two generators. A more general construction yields alternating quasicocycles on free product groups. To make this precise, let $A$ and $B$ be non-trivial groups and let $E$ be a Banach space endowed with a linear isometric action of $\Gamma=A\ast B$. For two alternating quasicocycles $f_A:A\longrightarrow E$, $f_B:B\longrightarrow E$ we have the \emph{split quasicocycle} $f_A\ast f_B:\Gamma\longrightarrow E$ given by
\begin{align*}
&(f_A\ast f_B)(a_1b_1\cdots a_nb_n):=\\&f_A(a_1)+a_1.f_B(b_1)+a_1b_1.f_A(a_2)+\ldots+a_1b_1a_2\cdots b_{n-1}a_n.f_B(b_n).
\end{align*}
This observation was made independently by A. Thom (\cite{Th}, Lemma 5.1). The construction induces a linear map
\[
\QZa(A,E)\times\QZa(B,E)\longrightarrow\h^2_\bb(\Gamma,E)
\]
which assigns to a pair of alternating quasicocycles $(f_A,f_B)$ the bounded cohomology class associated to the quasicocycle $f_A\ast f_B$.
The aim of this paper is to study the image of this map, referred to as the space of \emph{split classes}, under different assumptions on the target $E$. As above all splittings that we consider are assumed to have exactly two non-trivial factors. In Section 2 we obtain the following results:
\begin{thmn}[\ref{thm-fd}] Let $\Gamma=A\ast B$ be a finitely generated group for which the factor $A$ contains an infinite order element. The split classes form an infinite-dimensional subspace of $\h^2_\bb(\Gamma,E)$ for any finite-dimensional Banach $\Gamma$-module $E$.
\end{thmn}
\begin{thmn}[\ref{thm-lp}] Under the same assumptions the split classes form an infinite dimensional subspace of $\h^2_\bb(\Gamma,\ell^p(\Gamma))$ for all $1<p<\infty$, and also of $\h^2_\bb(\Gamma,\ell^1(\Gamma))$ if the factor $A$ is amenable.
\end{thmn}
The simple construction of split quasicocycles allows us to give rather short, linear-algebraic proofs of these facts. We note that infinite-dimensionality of the spaces in (ii) and (iii) follows from work of Hull-Osin (\cite{HO}, Corollary 1.7), and non-vanishing for $\ell^2$-coefficients was established before by Monod-Shalom (\cite{MS}, Corollary 7.9). Furthermore, a recently published construction of Bestvina-Bromberg-Fujiwara yields non-trivial quasi-cocycles for uniformly convex coefficients, their construction applies in particular to the spaces in (i) (see \cite{BBF}).

In the special case where the target is the trivial module $\R$ the construction yields \emph{split quasimorphisms}, which we study in Section 3. Almost all quasimorphisms that have been constructed for various groups are variations or generalizations of the counting quasimorphisms that were introduced by Brooks (see \cite{B}). It seems that even the following simple split quasimorphism on the free group of rank 2 has not yet been considered:
\begin{align*}
&f:\F_2=\langle a,b\rangle\longrightarrow\R\\
&f(a^{k_1}b^{k_2}\cdots a^{k_{n-1}}b^{k_n})=\#\left\{i\,|\,k_i>0\right\}-\#\left\{i\,|\,k_i<0\right\}.
\end{align*}
Two basic questions are whether a given quasimorphism is trivial, i.e. a bounded perturbation of a homomorphism, or homogenous, i.e. such that the restrictions to cyclic subgroups are homomorphisms. In our case we have
\begin{corn}[\ref{cor-qm-trivial}] Let $f=f_A\ast f_B$ be a split quasimorphism on $\Gamma=A\ast B$. The following are equivalent
\begin{enumerate}
\item[(i)]{$f$ is trivial}
\item[(ii)]{$f$ is homogenous}
\item[(iii)]{$f$ is a homomorphism}
\item[(iv)]{$f_A$ and $f_B$ are homomorphisms.}
\end{enumerate}
\end{corn}
An immediate consequence for the second bounded cohomology of $\Gamma$ is that for $\Gamma=A\ast B$ there is a linear embedding
\[
\frac{\QMa(A)}{\Hom(A,\R)}\times\frac{\QMa(B)}{\Hom(B,\R)}\hooklongrightarrow\h^2_\bb(\Gamma,\R).
\]
Here $\QMa$ stands for the spaces of alternating quasimorphisms. In Appendix A we use this result to give a short self-contained proof of the fact that the space $\ell^\infty$ embeds into the bounded cohomology $\h^2_\bb(\F_2,\R)$ of the free group.

The above corollary is deduced from the following statement which concerns the \emph{Gromov norm} of the cohomology class associated to a split quasimorphism $f$. This norm is equal to the infimum over the defects of all quasimorphisms at bounded distance from $f$.
\begin{thmn}[\ref{gromov-norm}]
Let $f=f_A\ast f_B$ be a split quasimorphism on $\Gamma=A\ast B$ with homogenization $\widehat{f}$ and corresponding cohomology class $\omega_f$. We have the equalities
\[
\|\omega_f\|=\tfrac12\,\deff\widehat{f}=\deff f=\max\{\deff f_A,\deff f_B\}.
\]
In particular, $f$ is a minimal defect representative for its class.
\end{thmn}
It is an open question whether the leftmost of these equalities holds for all quasimorphisms. Bavard proved that equality holds for a particular example of a counting quasimorphism (\cite{Ba}, Section 3.8). This theorem has the consequence that we can isometrically embed so called \emph{defect spaces} into $\h^2_\bb(\Gamma,\R)$. These are $\ell^\infty$-spaces equipped with an exotic norm. More precisely, we define the defect space $\D(\Gamma)$ of the group $\Gamma$ to be the space of alternating bounded functions $f:\Gamma\longrightarrow\R$ equipped with the defect $\dn{f}=\deff f=\sup_{g,h}|f(gh)-f(g)-f(g)|$ as a norm. This norm is equivalent to the usual supremum norm, so that $\D(\Gamma)$ is a Banach space equivalent to the alternating subspace of $\ell^\infty(\Gamma)$. Appendix B contains a self-contained summary of the basic properties of defect spaces. The above corollary and theorem imply
\begin{thmn}[\ref{thm-main}] For $\Gamma=A\ast B$ there is a linear isometric embedding
\[
\D(A)\infplus\D(B)\hooklongrightarrow\h^2_\bb(\Gamma,\R).
\]
\end{thmn}
By using properties of defect spaces for subgroups, we deduce
\begin{corn}[\ref{cor-main}] If the group $\Gamma$ admits a splitting $\Gamma=A\ast B$ such that $A$ contains an element of infinite order, then there is a linear isometric embedding
$\D(\Z)\hooklongrightarrow\h^2_\bb(\Gamma,\R)$. In particular, the Banach space $\h^2_\bb(\Gamma,\R)$ is non-separable.
\end{corn}
The approach to bounded cohomology developed in \cite{BM} yields an identification of the space $\h^2_\bb(\Gamma,\R)$ with a weak*-closed subspace of a certain $L^\infty$-space, so that $\h^2_\bb(\Gamma,\R)$ is non-separable whenever it is infinite-dimensional. We note that Calegari gave a simple argument showing that the space $\h^2_\bb(\F_2,\R)$ is non-separable (\cite{scl}, Example~2.62).

The two next corollaries are established using results of Antol\'{\i}n-Minasyan, Haglund-Wise and Agol which say that certain classes of groups (virtually) admit epimorphisms onto free groups (see \cite{AM}, \cite{HW}, \cite{Ag}).
\begin{corn}[\ref{cor-raag}] The space $\D(\Z)\infplus\D(\Z)$ embeds isometrically into $\h^2_\bb(\Gamma,\R)$ if the non-abelian group $\Gamma$ is
\begin{itemize}
\item[(i)] a subgroup of a right angled Artin group, or
\item[(ii)] the fundamental group of a compact special cube complex.
\end{itemize}
\end{corn}
For certain groups of type (i) infinite-dimensionality of $\h^2_\bb$ was proven by Behrstock-Charney (\cite{BC}, Theorem~5.2).
\begin{corn}[\ref{cor-3mf}]
The space $\D(\Z)\infplus\D(\Z)$ space embeds isometrically into $\h^2_\bb(\Gamma',\R)$ for a finite index subgroup $\Gamma'$, if the group $\Gamma$ is
\begin{itemize}
\item[(i)] word-hyperbolic and admits a proper and cocompact action on a CAT(0) cube complex, or
\item[(ii)] the fundamental group of a compact hyperbolic 3-manifold.
\end{itemize}
\end{corn}
This leads us to ask whether all word-hyperbolic groups (virtually) admit isometrically embedded defect spaces in their second bounded cohomology.

In Section 3.2 we address the question whether split quasicocycles can be defined on amalgamated products, and discuss some examples. It turns out that generalized split quasicocycles on $\Gamma=A\ast_C B$ are exactly the pullbacks of split quasicocycles on the largest free product quotient of $\Gamma$. We apply this fact to obtain
\begin{thmn}[\ref{thm-surface}]
Let $\Gamma_m$ be the fundamental group of the closed orientable surface of genus $m$. For $m,n\geq1$ there is a linear isometric embedding
\[
\D(\Gamma_m)\infplus\D(\Gamma_n)\hooklongrightarrow\h^2_\bb(\Gamma_{m+n},\R).
\]
\end{thmn}

In Section 3.3 we study actions of group automorphisms on split quasimorphisms and their cohomology classes. For a group $\Gamma$ we have a natural action of $\Out(\Gamma)$ on the bounded cohomology of $\Gamma$. For a class associated to a quasimorphism $f$ this action is induced by precomposition of $f$ with elements of $\Aut(\Gamma)$. If $f$ is a split quasimorphism $f_A\ast f_B$ on $\Gamma=A\ast B$ then it can be translated by an automorphism in a second way, namely by precomposition of the factor maps $f_A$,$f_B$. This induces an action of $\Out(\Gamma)$ on the subspace of $\h^2_\bb(\Gamma,\R)$ which is spanned by the split classes of all splittings of $\Gamma$. Using this split action we show

\begin{propn}[\ref{finite-factors}]
If the group $\Gamma$ admits a splitting $\Gamma=A\ast B$ with finite factors $A$,$B$ then the subspace of split classes in $\h^2_\bb(\Gamma,\R)$ is independent of the choice of a splitting.
\end{propn}
This has the consequence that the modular group $\PSL(2,\Z)=\Z/2\Z\ast\Z/3\Z$ has a unique split class (up to scaling), which turns out to be the class associated to the Rademacher function $\mathfrak{R}:\PSL(2,\Z)\longrightarrow\Z$.

We are not aware of any results concerning the stabilizers of the natural action of $\Out(\F_2)$ on $\h^2_\bb(\F_2,\R)$. We have the conjecture that an irreducible outer automorphism cannot stabilize a split class. In the reducible case we are able to show that there are invariant split classes:
\begin{thmn}[\ref{thm-fixed-point}] Let $f=f_A\ast f_B$ be a split quasimorphism on $\F_2=\langle a\rangle\ast\langle b\rangle$. The associated cohomology class $\omega_f$ is a fixed point of the outer automorphism $\tau_n$ given by $a\mapsto a$, $b\mapsto a^nb$, if and only if the function $f_A$ is $n$-periodic and the function $f_B$ is equal to zero.
\end{thmn}
\begin{corn}[\ref{cor-infinite-stabilizer}]
If $f_A:\Z\longrightarrow\R$ is bounded, alternating and periodic then the cohomology class $\omega_f$ associated to the split quasimorphism $f=f_A\ast 0$ has infinite stabilizer $\Stab_{\Out(\F_2)}(\omega_f).$
\end{corn}

In Section 3.4 we relate split quasimorphisms to counting quasimorphisms. A result of Grigorchuck says that any homogenous quasimorphism on the free group $\F_2$ can be expressed as a infinite linear combination of homogenized counting quasimorphisms. Due to the recursive nature of the proof it seems hopeless to compute the coefficients explicitly. We make the observation that split quasimorphisms can be written in a different way as explicit infinite sums of counting quasimorphisms.

In Section 4 we give a short discussion of quasi-representations (also known as $\varepsilon$-representations) $\Gamma\longrightarrow G$ that arise from free product splittings of a group $\Gamma$. These \emph{split quasi-representations} can take values in an arbitrary group $G$ endowed with a bi-invariant metric, such as a compact Lie group. They are obtained as a straightforward generalization of split quasicocycles $\Gamma\longrightarrow E$, in whose construction we do not use the fact that the target group $(E,+)$ is commutative. Our result in this context is
\begin{thmn}[\ref{thm-qr}] Let $\Gamma=A\ast B$ and let $G=(G,d)$ be a group with a bi-invariant metric without $\varepsilon$-small subgroups. For bounded alternating maps $\mu_A:A\longrightarrow G$, $\mu_B:B\longrightarrow G$ with
\[
\delta:=\max\{\|\mu_A\|_\infty,\|\mu_B\|_\infty\}\leq\frac{\varepsilon}{2}.
\]
the split quasi-representation $\mu=\mu_A\ast\mu_B:\Gamma\longrightarrow G$ is non-trivial, in the sense that $d(\mu,\rho)\geq\delta$ for every representation $\rho:\Gamma\longrightarrow G$.
\end{thmn}
This result was obtained by the author in his master thesis (see \cite{R}). Burger-Ozawa-Thom have used the construction in the context of Ulam stability (see \cite{BOT}), in particular they have shown that for all $n\geq1$ there are split quasi-representations $\F_2\longrightarrow U(n)$ which have dense image.
\subsection*{Acknowledgement} The author is thankful to Marc Burger, Theo B\"uhler and Beatrice Pozzetti for their remarks on preliminary versions of this paper.
\subsection{Definitions and basic results}

Let $\Gamma$ be a group. A map $f:\Gamma\longrightarrow\R$ is called a \emph{quasimorphism} if there exists $C>0$ such that
\[
|f(gh)-f(g)-f(h)|<C,\qquad\forall\,g,h\in\Gamma.
\]
The defect of a quasimorphism $f$ is defined to be
\[
\deff f:=\sup_{g,h\in\Gamma}|f(gh)-f(g)-f(h)|.
\]
If $E$ is a Banach $\Gamma$-module, i.e. a Banach space endowed with a linear isometric action of $\Gamma$, we say that $f:\Gamma\longrightarrow E$ is a \emph{quasicocycle} if there exists $C>0$ such that
\[
\|f(gh)-f(g)-g.f(h)\|_E<C,\qquad\forall\,g,h\in\Gamma.
\]
The defect of such an $f$ is defined accordingly.
We denote by $\QZ(\Gamma,E)$ the space of quasicocycles with values in $E$, and by $\QM(\Gamma)=\QZ(\Gamma,\R)$ the space of quasimorphisms.
Recall that the group cohomology $\h^*(\Gamma,E)$ is computed by the bar complex
\[
0\xrightarrow{} E\xrightarrow{}\Map(\Gamma,E)\xrightarrow{\partial^1}\Map(\Gamma^2,E)\xrightarrow{\partial^2}\Map(\Gamma^3,E)\xrightarrow{\partial^3}\dots
\]
which has the coboundary operators
\begin{align*}
\partial^k &f(g_1,\dots,g_{k+1}):=g_1.f(g_2,\dots,g_{k+1})\\&+\sum_{i=1}^k(-1)^i f(g_1,\dots,g_ig_{i+1},\dots,g_{k+1})+(-1)^{k+1}f(g_1,\dots,g_{k+1}).
\end{align*}
The bounded cohomology $\h^*_\bb(\Gamma,E)$ is the cohomology of the subcomplex of bounded maps, which we call the \emph{bounded bar complex}:
\[
0\xrightarrow{} E\xrightarrow{}\ell^\infty(\Gamma,E)\xrightarrow{\partial^1}\ell^\infty(\Gamma^2,E)\xrightarrow{\partial^2}\ell^\infty(\Gamma^3,E)\xrightarrow{\partial^3}\dots
\]
We denote the cocycles in these complexes by $\ZZ^*(\Gamma,E)$ and $\ZZ^*_\bb(\Gamma,E)$ respectively.
The $1$-coboundary of a quasicocycle  $f:\Gamma\longrightarrow E$ as introduced above is given by $\partial^1f(g,h)=f(g)+g.f(h)-f(gh)$, so $f$ is almost a cocycle in the bar complex, and since $\partial^2\partial^1=0$ we have that $\partial^1 f$ is a 2-cocycle in the bounded bar complex. We denote by $\omega_f:=[\partial^1\!f]_\bb$ the corresponding bounded cohomology class. To say it short, we have a linear map
\[
\QZ(\Gamma,E)\longrightarrow\h^2_\bb(\Gamma,E),\quad f\mapsto\omega_f
\]
whose image $\eh^2_\bb(\Gamma,E)$ is equal to the kernel of the comparison map $\h^2_\bb(\Gamma,E)\longrightarrow\h^2(\Gamma,E)$. It is straightforward to check that a quasicocycle is in the kernel of the above map if and only if it admits a decomposition $f=\varphi+\beta$ into an actual cocycle $\varphi\in\ZZ^1(\Gamma,E)$ and a bounded perturbation $\beta\in\ell^\infty(\Gamma,E)$. These are called \emph{trivial} quasicocycles, the decomposition is called a \emph{trivialization}. A trivialization is unique if and only if the module $E$ is trivial. In general the components of two trivializations may differ by \emph{inner cocycles}
\[
\iota_v:\Gamma\longrightarrow E,\quad g\mapsto g.v-v,\quad v\in E.
\]
These are the $1$-coboundaries in the bounded bar complex. With this terminology, the space $\h^1_\bb(\Gamma,E)$ can be described as the quotient of the bounded $1$-cocycles modulo the inner cocycles. Under fairly general conditions every bounded 1-cocycle is inner:
\begin{proposition}[\cite{Mo}, Proposition~6.2.1, Corollary~7.5.11]
\label{prop-h1b} For a group $\Gamma$ and a Banach $\Gamma$-module $E$ we have $\h^1_\bb(\Gamma,E)=0$
\begin{itemize}
\item[(i)] if $E$ is reflexive as a Banach space, or
\item[(ii)] if $\Gamma$ is amenable and $E$ is a coefficient module.
\end{itemize}
\end{proposition}
We will not recall the definition of a coefficient module here, for our purposes it is sufficient to know that the $\Gamma$-modules $\ell^p(\Gamma)$, $1\leq p\leq\infty$, are coefficient modules (\cite{Mo}, Examples~1.2.3).

The spaces $\h^k_\bb(\Gamma,E)$ carry a quotient semi-norm coming from the norms of the $\ell^\infty$-spaces in the bounded bar complex. For $k=2$ this is a proper norm which turns $\h^2_\bb(\Gamma,E)$ into a Banach space (see \cite{I}). Calculating this $\emph{Gromov norm}$ for a cocycle of the form $\omega_f$ amounts to finding the infimum of the defects over all the quasicocycles at bounded distance from $f$:
\begin{align*}
\|\omega_f\|&=\inf\left\{\deff \overline{f}\,\,|\,\,\overline{f}\in\QZ(\Gamma,E)\mbox{ such that }\,f-\overline{f}\mbox{ is bounded}\right\}\\
&=\inf\left\{\deff (f+\beta)\,\,|\,\,\beta\in\ell^\infty(\Gamma,E)\right\}.
\end{align*}
In the case of trivial coefficients $E=\R$ the Gromov norm is related to the notion of a \emph{homogenous} quasimorphism. This is a quasimorphism $f:\Gamma\longrightarrow\R$ for which
\[
f(g^n)=n\cdot f(g),\quad\forall g\in\Gamma\;\forall n\in\Z,
\]
which is to say that $f$ restricts to a homomorphism on every cyclic subgroup. We write $\HQM(\Gamma)$ for the corresponding subspace of $\QM(\Gamma)$. Homogenous quasimorphisms are conjugacy invariant, and for each quasimorphism $f$, there is a unique $\widehat{f}\in\HQM(\Gamma)$ at bounded distance from $f$. This \emph{homogenization} of $f$ is given by
\[
\widehat{f}(g)=\lim_{n\rightarrow\infty}\frac{f(g^n)}{n}.
\]
and the assignment $f\mapsto\widehat{f}$ defines a projection $\QM(\Gamma)\longrightarrow\HQM(\Gamma)$.
The following result of Bavard (\cite{Ba}, Section 3.6) provides a lower bound for the Gromov norm of the class of a quasimorphism
\begin{theorem}
\label{thm-bavard}
For any group $\Gamma$ and any quasimorphism $f:\Gamma\longrightarrow\R$ we have
\[
\|\omega_f\|\geq\tfrac12\cdot\deff\widehat{f}.
\]
\end{theorem}
For later reference we also list the following result, a proof of which can be found in Huber's thesis (\cite{Hu}, Theorem~2.14):
\begin{theorem}
\label{th-huber}
An epimorphism $\varphi:\Gamma\longrightarrow\Gamma'$ between countable groups induces an isometric embedding
\[
\varphi^*:\h^2_\bb(\Gamma',\R)\hooklongrightarrow\h^2_\bb(\Gamma,\R).
\] 
\end{theorem}
For an automorphism $\tau\in\Aut(\Gamma)$ there is an induced isomorphism $\tau^*$ in bounded cohomology with trivial coefficients, so that for each $k\geq0$ we have an action of $\Aut(\Gamma)$ on $\h^k_\bb(\Gamma,\R)$, given by $\tau.\omega=(\tau^*)^{-1}\omega$. For inner automorphisms this action is trivial (\cite{Mo}, Lemma~8.7.2), so that we have an induced action of $\Out(\Gamma)$. This action preserves the semi-norm mentioned above, so that in particular we have a linear isometric action of $\Out(\Gamma)$ on the Banach space $\h^2_\bb(\Gamma,\R)$. We will call this the natural action of $\Out(\Gamma)$ on $\h^2_\bb$.

\section{Split quasicocycles}
Let $\Gamma$ be a group and let $E$ be a Banach $\Gamma$-module. We write
\[
\QZa(\Gamma,E):=\{f\in\QZ(\Gamma,E)\,|\,f(g)+g.f(g^{-1})=0\}
\]
for the space of \emph{alternating} quasicocycles. Assume that we have a splitting $\Gamma=A\ast B$. Through the embeddings $A,B\hooklongrightarrow\Gamma$ the space $E$ is equipped with an $A$- and $B$-module structure. We repeat here in some more detail the construction of a split quasicocycle on $\Gamma$, which we introduced in \cite{R}. For that purpose let $f_A\in\QZa(A,E)$ and $f_B\in\QZa(B,E)$. We define a map
\[
f_A\ast f_B:\Gamma\longrightarrow E
\]
as follows: For an element $1\neq g\in\Gamma$ we consider its normal form
\[
g=a_1b_1a_2b_2\cdots a_n b_n,
\]
in which $a_i\in A$, $b_i\in B$ and only $a_1$ or $b_n$ are possibly trivial. We set
\begin{align*}
&(f_A\ast f_B)(g):=\\&f_A(a_1)+a_1.f_B(b_1)+a_1b_1.f_A(a_2)+\ldots+a_1b_1a_2\cdots b_{n-1}a_n.f_B(b_n).
\end{align*}
\begin{proposition}
\label{prop-qc}
The map $f=f_A\ast f_B$ is an alternating quasicocycle on $\Gamma$ with $\deff f=\max\{\deff f_A,\deff f_B\}$. The induced linear map
\[
\QZa(A,E)\times\QZa(B,E)\longrightarrow\QZa(\Gamma,E),\quad (f_A,f_B)\mapsto f
\]
extends the natural isomorphism
\[
\ZZ^1(A,E)\times\ZZ^1(B,E)\longrightarrow\ZZ^1(\Gamma,E).
\]
\end{proposition}

\begin{proof}
The fact that the map $f$ is alternating follows immediately from the corresponding property of the factor maps. We show that $f$ is indeed a quasicocycle. Let $g,h\in\Gamma$. If $g$ ends with an $A$-letter and $h$ begins with $B$-letter or vice versa, then $\partial f(g,h)=0$ since the normal form of $gh$ equals the concatenation of the normal forms of $g$ and $h$. If the normal forms are $g=g'a$ and $h=a^{-1}h'$ then
\begin{align*}
\partial f(g,h)&=f(g'h')-f(g'a)-g'a.f(a^{-1}h')\\
&=f(g'h')-f(g')-g'.(f(a)+a.f(a^{-1}))-g'.f(h')\\
&=\partial f(g',h')
\end{align*}
since the quasicocycle $f_A$ is alternating. The same holds for $B$-letters. So we may assume that $g=g'a_1$ and $h=a_2h'$ with $a_1a_2\neq1$ (or likewise with $B$-letters). In this case we have
\begin{align*}
\partial f(g,h)=&f(g'a_1a_2h')-f(g'a_1)-g'a_1.f(a_2h')\\
=&f(g')+g'.f(a_1a_2)+g'a_1a_2.f(h')\\&-f(g')-g'.f(a_1)-g'a_1.f(a_2)-g'a_1a_2.f(h')\\
=&g'.(f(a_1a_2)-f(a_1)-a_1.f(a_2))\\
=&g'.\partial f_A(a_1,a_2),
\end{align*}
so that $\|\partial f(g,h)\|_E=\|\partial f_A(a_1,a_2)\|_E\leq\deff f_A$. Hence $f$ is a quasicocycle with the defect indicated above. 
\end{proof}
Note that the quasicocycle $f_A\ast f_B$ is an actual cocycle if and only if $f_A$ and $f_B$ are both cocycles. In particular we have $\iota_v^A\ast\iota_v^B=\iota_v^\Gamma$, where the inner cocycles are defined on the groups indicated in the superscript.

We refer to bounded cohomology classes of the form $\omega_f$, where $f$ is a split quasicocycle for the group $\Gamma=A\ast B$, as \emph{split classes}.
\subsection{Finite-dimensional coefficients}
The aim of this subsection is to prove
\begin{theorem}
\label{thm-fd}
Let $\Gamma$ be a finitely generated group with a splitting $\Gamma=A\ast B$, such that $A$ contains an element of infinite order. Then for any finite dimensional Banach $\Gamma$-module $E$ the split classes form an infinite dimensional subspace of $\h^2_\bb(\Gamma,E)$.
\end{theorem}
\begin{corollary} For a non-abelian free group $\F$ the split classes form an infinite dimensional subspace of $\h^2_\bb(\F_,E)$ for any finite-dimensional Banach $\F$-module $E$.
\end{corollary}
\begin{proof}[Proof of Theorem~\ref{thm-fd}]
Let
\[
\mathcal{L}:=\lia(A,E)\times\lia(B,E)\subseteq\QZa(A,E)\times\QZa(B,E)
\]
and let
\[
\mathcal{B}:=\left\{(f_A,f_B)\in\mathcal{L}\,|\,f_A\ast f_B\mbox{ is bounded}\right\}.
\]
The construction of spilt quasicocycles yields a map
\[
\Psi:\mathcal{L}\longrightarrow\h^2_\bb(\Gamma,E).
\]
Note that $\mathcal{B}\subseteq\ker\Psi$ since bounded quasicocycles are trivial. The statement of the theorem follows from the following two facts
\begin{lemma}
\label{lem-fd-1}
The space $\ker\Psi\,/\,\mathcal{B}$ has finite dimension.
\end{lemma}
\begin{lemma}
\label{lem-fd-2}
The space $\mathcal{L}\,/\,\mathcal{B}$ has infinite dimension.
\end{lemma}
\noindent Indeed, we have
\[
\im\Psi\quad\cong\quad\mathcal{L}\,/\ker\Psi\quad\cong\quad(\mathcal{L}\,/\,\mathcal{B})\,/\,(\ker\Psi\,/\,\mathcal{B}),
\]
where the latter space has infinite dimension.
\end{proof}
\begin{proof}[Proof of Lemma~\ref{lem-fd-1}] If $(f_A,f_B)\in\ker\Psi$ then the quasicocycle $f:=f_A\ast f_B$ has a trivialization $f=\varphi+\beta$, where $\varphi\in\ZZ^1(\Gamma,E)$ and $\beta\in\ell^\infty(\Gamma,E)$.
If $f=\varphi'+\beta'$ is another trivialization then the cocycle $\varphi-\varphi'=\beta'-\beta$ is bounded. This means that we can assign to $(f_A,f_B)$ a cocycle which is well defined up to addidtion of a bounded cocycle. Hence we have a map
\[
\ker\Psi\longrightarrow\ZZ^1(\Gamma,E)\,/\,\ZZ^1_\bb(\Gamma,E).
\]
The kernel of this map is equal to the space $\mathcal{B}$ and so we have an embedding
\[
\ker\Psi\,/\,\mathcal{B}\hooklongrightarrow\ZZ^1(\Gamma,E)\,/\,\ZZ^1_\bb(\Gamma,E).
\]
Since $\Gamma$ is finitely generated and $E$ is finite-dimensional the space $\ZZ^1(\Gamma,E)$ is finite dimensional as well, the claim follows.
\end{proof}
\begin{proof}[Proof of Lemma ~\ref{lem-fd-2}] Fix a non-zero vector $v\in E$, an element $a\in A$ of infinite order and a non-trivial element $b\in B$. For a prime number $p$ and $n\geq0$ we define words $w_{p,n}\in\Gamma$ by
\begin{align*}
w_{p,0}&:=1\\
w_{p,n}&:=ba^pba^{p^2}\cdots ba^{p^n},\quad n\geq1
\end{align*}
\textit{Claim:} For all prime numbers $p$ we can choose a bounded map $f_A^p\in\lia(A,E)$ such that the split quasicocycles $f^p:=f_A^p\ast0$ satisfy
\begin{align*}
f^p(w_{p,n})&=n\cdot v&\forall p,n\\
f^p(w_{q,n})&=0&\forall q\neq p\,\forall n.
\end{align*}
\noindent\textit{Proof of the claim}. For $g=a^{p^n}$, $n\geq1$, define
\[
f_A^p(g)=(b\,w_{p,n-1})^{-1}.v
\]
and extend to negative powers $a^{-p^i}$ in the way needed to make $f_A^p$ alternating. For all $g\in A$ which are not of the form $g=a^{\pm p^i}$ set $f_A^p(g)=0$. Using the construction of split quasicocycles we obtain
\begin{align*}
f^p(w_{p,n})&=f^p_A(w_{p,n-1})+(w_{p,n-1}b).f_A^p(w_{p,n})\\
&=f^p_A(w_{p,n-1})+v\\
&=f^p_A(w_{p,n-2})+2v\\
&=\ldots=n\cdot v.
\end{align*}
The property $f^p(w_{q,n})=0$ holds by construction and thus the claim is established.
We finally show that the intersection
\[
\mathcal{B}\,\cap\,\mathrm{span}\left\{(f_A^p,0)\,|\,\mbox{$p$ prime}\right\}
\]
of subspaces of $\mathcal{L}$ is empty, which implies the statement of the lemma. So assume that $f=\sum_j\lambda_jf^{p_j}$ is a bounded quasicocycle. Evaluating at $w_{p_j,n}$ yields the equation $\lambda_jf(w_{p_j,n})=n\cdot v$, whence $\lambda_j=0$ for all $j$.
\end{proof}

\subsection{$\ell^p$-coefficients}
For a countable group $\Gamma$ and $1\leq p\leq\infty$ we endow $\ell^p(\Gamma)$ with the usual left-regular action. That is, for $\chi\in\ell^p(\Gamma)$ and $g\in\Gamma$ we set $(g.\chi)(h)=\chi(g^{-1}h)$.
\begin{theorem}
\label{thm-lp}
Let $\Gamma$ be a countable group with a splitting $\Gamma=A\ast B$, such that $A$ contains an element of infinite order. For $1<p<\infty$ the split classes form an infinite-dimensional subspace of $\h^2_\bb(\Gamma,\ell^p(\Gamma))$. If the factor $A$ is amenable then the same holds for $p=1$.
\end{theorem}
\begin{corollary}
For a non-abelian free group $\F$ the split classes form an infinite dimensional subspace of $\h^2_\bb(\F,\ell^p(\F))$ for $1\leq p<\infty$.
\end{corollary}
We first establish the fact
\begin{lemma}
\label{lem-lp-1}
Let $\Gamma=A\ast B$ be a splitting, and let $f:\Gamma\longrightarrow E$ be a trivial quasicocycle which is bounded on the free factors $A$ and $B$. If either $E$ is reflexive, or $A$ is amenable and $E$ is a coefficient module, then $f$ has a trivialization of the form
\[
f=0\ast\varphi_B+\beta
\]
for some $\varphi_B\in\ZZ^1(B,E)$ and some $\beta\in\ell^\infty(\Gamma,E)$.
\end{lemma}
\begin{proof}
Let $f=\varphi+\beta$ be a trivialization of $f$. By assumption, the cocycle $\varphi$ splits as  $\varphi=\varphi_A\ast\varphi_B$ into bounded cocycles on the factors. By Proposition~\ref{prop-h1b} we have that $\h^1_\bb(A,E)=0$, so $\varphi_A=\iota^A_v$ for some $v\in E$. Write
\begin{align*}
f&=\iota^A_v\ast\varphi_B+\beta\\
&=(\iota^A_v\ast\varphi_B-\iota^A_v\ast\iota^B_v)+(\iota^A_v\ast\iota^B_v+\beta)\\
&=0\ast(\varphi_B-\iota_v^B)+(\iota^\Gamma_v+\beta)
\end{align*}
which is, up to renaming, a trivialization of the desired type.
\end{proof}
\begin{proof}[Proof of Theorem~\ref{thm-lp}]
We construct an embedding
\[
\ell^p(\Gamma)\hooklongrightarrow\lia(A,\ell^p(\Gamma)),\quad\xi\mapsto r_\xi
\]
such that the split quasicocycle $r_\xi\ast0$ is trivial if and only if $\xi=0$.
We begin with fixing an infinite order element $a\in A$ and a non-trivial element $b\in B$. Let $w_n\in\Gamma$ be the sequence defined by
\begin{align*}
w_0&=1\\
w_1&=a\\
w_n&=aba^2ba^3b\cdots a^{n-1}ba^n,\quad n\geq2.
\end{align*}
For $\xi\in\ell^p(\Gamma)$ we define the bounded map $r_\xi\in\lia(A,\ell^p(\Gamma))$ as follows: Set
\[
r_\xi(a^n)=w_{n-1}^{-1}.\xi,\quad n\geq1,
\]
and extend to negative powers of $a$ in the way needed to make $r_\xi$ alternating. For all $g\in A$ that are not powers of $a$ set $r_\xi(g)=0$. Assume that the split quasicocycle $f_\xi:=r_\xi\ast 0$ is trivial. Since $\ell^p$-spaces are reflexive for $1<p<\infty$, and since we assume $A$ to be amenable in case $p=1$, Lemma~\ref{lem-lp-1} yields a trivialization
\[
f_\xi=0\ast\varphi_B+\beta.
\]
We evaluate at this equation at $w_n$, where we write $\zeta:=\varphi_B(b)\in\ell^p(\Gamma)$. By construction we have $f_\xi(w_n)=n\cdot\xi$, so
\[
n\cdot\xi=w_1.\zeta+\ldots+w_{n-1}.\zeta+\beta(w_n).
\]
This is an equation of functions in $\ell^p(\Gamma)$ which we evaluate further at $g\in\Gamma$ to obtain
\[
n\cdot\xi(g)=\zeta(w_1^{-1}g)+\ldots+\zeta(w_{n-1}^{-1}g)+\beta(w_n)(g).
\]
Using the H\"older inequality we obtain
\begin{align*}
&|\zeta(w_1^{-1}g)|+\ldots+|\zeta(w_{n-1}^{-1}g)|\\
&\leq (n-1)^{1-1/p}\left(|\zeta(w_1^{-1}g)|^p+\ldots+|\zeta(w_{n-1}^{-1}g)|^p\right)^{1/p}\\
&\leq (n-1)^{1-1/p}\cdot\|\zeta\|_{\ell^p(\Gamma)}\\
&\leq n^{1-1/p}\cdot\|\zeta\|_{\ell^p(\Gamma)}.
\end{align*}
Furthermore we have
\[
|\beta(w_n)(g)|\leq\|\beta(w_n)\|_{\ell^p(\Gamma)}\leq\|\beta\|_{\ell^\infty(\Gamma,\ell^p(\Gamma))},
\]
and hence
\[
n\cdot|\xi(g)|\leq n^{1-1/p}\cdot\|\zeta\|_{\ell^p(\Gamma)}+\|\beta\|_{\ell^\infty(\Gamma,\ell^p(\Gamma))}.
\]
Dividing both sides by $n$ and letting $n$ tend to infinity finally yields $\xi=0$. Thus we have shown that the composition
\[
\ell^p(\Gamma)\xrightarrow{\quad I\quad}\lia(A,\ell^p(\Gamma))\times\lia(B,\ell^p(\Gamma))\longrightarrow\h^2_\bb(\Gamma,\ell^p(\Gamma))
\]
with $I(\xi)=(r_\xi,0)$ is an embedding. The statement follows.
\end{proof}
\begin{remark} There are certain types of coefficients for which all split classes are trivial. Indeed, for any group $\Gamma$ one has $\h^*_\bb(\Gamma,\ell^\infty(\Gamma))=0$, and the same holds more generally for relatively injective coefficient modules (\cite{Mo}, Proposition 7.4.1). Furthermore one can check that for $\Gamma=A\ast B$ the split classes in $\h^2_\bb(\Gamma,\ell^\infty(\Gamma)/\R)$ vanish. This is unfortunate, since this space is isomorphic (\cite{Mo}, Proposition 10.3.2) to the poorly understood space $\h^3_\bb(\Gamma,\R)$, which is known to be infinite dimensional for free groups (\cite{So}, Theorem 3).
\end{remark}
\begin{question} Is it true that for every reflexive $\F_2$-Banach module $E$ the split classes form an infinite dimensional subspace of $\h^2_\bb(\F_2,E)$ ?
\end{question}
\section{Split quasimorphisms}
\subsection{Embedding defect spaces}

Recall that a split quasimorphism $f=f_A\ast f_B$ on $\Gamma=A\ast B$ is defined by
\[
f(a_1b_1\cdots a_nb_n)=f_A(a_1)+f_B(b_1)+\dots+f_A(a_n)+f_B(b_n),
\]
where $f_A$ and $f_B$ are alternating quasimorphisms on the factors. We determine the homogenization $\widehat{f}$ and calculate the Gromov norm of the bounded cohomology class $\omega_f$. A non-trivial element of $\Gamma$ is called \emph{cyclically reduced} if its normal form begins with an $A$-letter and ends with a $B$-letter or vice versa. Note that this is not the standard use of the terminology.
\begin{proposition}
\label{prop-homog} The homognization of a split quasimorphism $f=f_A\ast f_B$ on $\Gamma=A\ast B$ is given by
\[
\widehat{f}(g)=
   \left\{
     \begin{array}{lll}
       \widehat{f_A}(g) & , & \mbox{if $g\in A$}\\
       \widehat{f_B}(g) & , & \mbox{if $g\in B$}\\\\
       f(g') & , &\mbox{if $g\not\in A\cup B$}\\
       & &\mbox{where $g'$ is any cyclically reduced conjugate of $g$.}
     \end{array}
   \right.
\]
\end{proposition}
\begin{proof}
If $g\in A$ then $f(g^n)=f_A(g^n)$, so $\widehat{f}(g)=\widehat{f_A}(g)$ and likewise for $g\in B$. So assume that $g\not\in A\cup B$. For a cyclically reduced conjugate $g'$ of $g$ we have $f(g'^{\,n})=n\cdot f(g')$ by construction of $f$. By conjugacy invariance of $\widehat{f}$ we obtain $\widehat{f}(g)=\widehat{f}(g')=f(g')$.
\end{proof}
\begin{lemma}
\label{lem-deff-homog}
Let $\widehat{f}$ be the homogenization of $f=f_A\ast f_B$. We have
\[
\deff{\widehat{f}}\geq2\cdot\max\{\deff{f_A},\deff{f_B}\}.
\]
\end{lemma}
\begin{proof} We may assume that $\deff{f_A}\geq\deff{f_B}$, and further that we can choose $a\in A$ with $a^2\neq1$. (Otherwise $\QMa(A)=0$, hence $\deff f_A=\deff f_B=\deff\widehat{f}=0$ and the statement is obvious). Let $b\in B$ and $a_1,a_2\in A$ be non-trivial elements. We consider the words
\begin{align*}
g&=ab\,a_2b\,a_1b^{-1}a^{-1},\\
h&=a^{-1}b^{-1}a_2b\,a_1b\,a.
\end{align*}
for which we have
\[
\widehat{f}(g)=\widehat{f}(h)=\widehat{f}(a_2ba_1)=f(a_1a_2b)=f_A(a_1a_2)+f_B(b)
\]
by conjugation invariance of $\widehat{f}$ and Proposition~\ref{prop-homog}. Furthermore,
\begin{align*}
\widehat{f}(gh)&=\widehat{f}(aba_2ba_1b^{-1}a^{-2}b^{-1}a_2ba_1ba)\\
&=f(a^2ba_2ba_1b^{-1}a^{-2}b^{-1}a_2ba_1b)\\
&=2(f_A(a_1)+f_A(a_2)+f_B(b)),
\end{align*}
and hence
\[
\widehat{f}(g)+\widehat{f}(h)-\widehat{f}(gh)=2\cdot(f_A(a_1)+f_A(a_2)-f_A(a_1a_2))
\]
which implies
\begin{align*}
\deff{\widehat{f}}&\geq2\cdot\sup\{f_A(a_1)+f_A(a_2)-f_A(a_1a_2)\,|\,a_1,a_2\in A, a_1,a_2\neq1\}\\
&=2\cdot\deff{f_A}\\
&=2\cdot\max\{\deff{f_A},\deff{f_B}\}.\qedhere
\end{align*}
\end{proof}
\begin{theorem}
\label{gromov-norm}
Let $f=f_A\ast f_B$ be a split quasimorphism with corresponding cohomology class $\omega_f=[\partial f]_\bb$. We have
\[
\|\omega_f\|=\tfrac12\,\deff{\widehat{f}}=\deff{f}=\max\{\deff{f_A},\deff{f_B}\}.
\]
In particular, $f$ is a minimal defect representative for its class.
\end{theorem}
\begin{proof} We have
\[
\max\{\deff{f_A},\deff{f_B}\}\leq\tfrac12\,\deff{\widehat{f}}\leq\|\omega_f\|\leq\deff{f}=\max\{\deff{f_A},\deff{f_B}\}
\]
by Theorem~\ref{thm-bavard} and Lemma~\ref{lem-deff-homog}.
\end{proof}
\begin{corollary}
\label{cor-qm-trivial}
For a split quasimorphism $f=f_A\ast f_B$ the following are equivalent:
\begin{enumerate}
\item[(i)]{$f$ is trivial}
\item[(ii)]{$f$ is a homomorphism}
\item[(iii)]{$f$ is homogenous}
\item[(iv)]{$f_A$ and $f_B$ are homomorphisms}
\end{enumerate}
\end{corollary}
\begin{corollary} For $\Gamma=A\ast B$ the kernel of the linear map
\[
\QMa(A)\times\QMa(B)\longrightarrow\h^2_\bb(\Gamma,\R),\quad(f_A,f_B)\mapsto\omega_f
\]
with $f=f_A\ast f_B$, is equal to $\Hom(A,\R)\times\Hom(B,\R)$.
\end{corollary}
Since there are no bounded real-valued homomorphisms the last statement yields a linear embedding
\[
\lia(A)\oplus\lia(B)\hooklongrightarrow\h^2_\bb(\Gamma,\R).
\]
By renorming the spaces $\lia$ in a suitable way, this embedding can be made isometric. We define the defect space $\D(\Gamma)$ of a group $\Gamma$ to be the space of bounded alternating functions on $\Gamma$, equipped with the defect $\dn{\cdot}:=\deff(\cdot)$ as a norm. Appendix B contains a self-contained discussion of the basic properties of these spaces. Using this definition together with Theorem~\ref{gromov-norm} we obtain
\begin{theorem}
\label{thm-main}
For a group $\Gamma=A\ast B$ there is a linear isometric embedding
\[
\D(A)\infplus\D(B)\hooklongrightarrow\h^2_\bb(\Gamma,\R)
\]
which maps the pair $(f_A,f_B)$ to the bounded cohomology class $\omega_f$ of the split quasimorphism $f=f_A\ast f_B$.
\end{theorem}
\noindent Here the notation $\infplus$ stands for the direct sum equipped with the $\max$-norm.
\begin{corollary}
\label{cor-main}
If $\Gamma=A\ast B$ such that $A$ contains an infinite order element, then there is a linear isometric embedding
\[
\D(\Z)\hooklongrightarrow\h^2_\bb(\Gamma,\R).
\]
In particular, the Banach space $\h^2_\bb(\Gamma,\R)$ is non-separable.
\end{corollary}
\begin{proof} Let $\langle a\rangle\cong\Z$ be an infinite cyclic subgroup of $A$. The space $\D(\langle a\rangle)$ is non-separable by Corollary~\ref{norm-equiv}, and by Proposition~\ref{subgroup-emb} and the theorem we have the embeddings
\[
\D(\langle a\rangle)\hooklongrightarrow\D(A)\hooklongrightarrow\D(A)\infplus\D(B)\hooklongrightarrow\h^2_\bb(\Gamma,\R).\qedhere
\]
\end{proof}
\begin{corollary} If the group $\Gamma$ admits an epimorphism $\Gamma\longrightarrow\F_2$ then there is a linear isometric embedding
\[
\D(\Z)\infplus\D(\Z)\hooklongrightarrow\h^2_\bb(\Gamma,\R).
\]
\end{corollary}
\begin{proof} Apply Theorem~\ref{th-huber}.
\end{proof}
In what follows we apply the last statement to certain classes of groups that have been shown to (virtually) surject onto free groups.
\begin{corollary}
\label{cor-raag}
If the non-abelian group $\Gamma$ is
\begin{itemize}
\item[(i)] a subgroup of a right-angled Artin group, or
\item[(ii)] the fundamental group of a compact special cube complex,
\end{itemize}
then there is a linear isometric embedding $\D(\Z)\infplus\D(\Z)\hooklongrightarrow\h^2_\bb(\Gamma,\R)$.
\end{corollary}
\begin{proof} The group surjects onto $\F_2$ if it is of type (i) (\cite{AM}, Corollary~1.6), and every group of type (ii) is also of type (i) (\cite{HW}, Theorem~1.1).
\end{proof}
\begin{corollary}
\label{cor-3mf}
If the group $\Gamma$
\begin{itemize}
\item[(i)] is word-hyperbolic and admits a proper and cocompact action on a CAT(0) cube complex, or
\item[(ii)] is the fundamental group of a compact hyperbolic 3-manifold,
\end{itemize} then there is a finite index subgroup $\Gamma'<\Gamma$ which admits a linear isometric embedding $\D(\Z)\infplus\D(\Z)\hooklongrightarrow\h^2_\bb(\Gamma',\R)$.
\end{corollary}
\begin{proof} A group of the type (i) has a finite index subgroup which is of type (ii) of the previous corollary (\cite{Ag}, Theorem~1.1). Moreover, every group of type (ii) is of type (i) (\cite{BW}, Theorem~5.3).
\end{proof}
By a result of Epstein-Fujiwara non-elementary word-hyperbolic groups have infinite-dimensional $\h^2_\bb$ (\cite{Fu}, Theorem 1.1), thus we can ask whether every such group (virtually) admits an embeddeded defect space in its second bounded cohomology:
\begin{question}
Does every non-elementary word-hyperbolic group $\Gamma$ have a finite index subgroup $\Gamma'$ such that the space $\h^2_\bb(\Gamma',\R)$ contains an isometrically embedded copy of
\begin{itemize}
\item[(i)] $\D(\Z)$ ?
\item[(ii)] $\D(\Z)\infplus\D(\Z)$ ?
\end{itemize}
\end{question}
\subsection{Amalgamated products}
It is natural to ask whether the construction of split quasicocycles generalizes to a construction for amalgamated products $\Gamma=A\ast_C B$. Generalized counting quasimorphisms for such groups have been constructed by Fujiwara (see \cite{Fu2}). If one tries to define quasicocycles $f=f_A\ast_Cf_B$ on $\Gamma$ by using the normal form for amalgams, it turns out that the required compatibility between $f_A\in\QZa(A,E)$, $f_B\in\QZa(B,E)$ is so strong that the map $f$ actually descends to a free product quotient of $\Gamma$, more precisely to the largest natural such quotient:
\[
\pi:A\ast_C B\longrightarrow A/\langle\langle C\rangle\rangle\ast B/\langle\langle C\rangle\rangle.
\]
Here $\langle\langle C\rangle\rangle$ stands for the normal closure of $C$ in $A$ and $B$ respectively. In other words, quasicocycles constructed this way are merely pullbacks of split quasicocycles on free products. In the case of quasimorphisms we know (Theorem~\ref{th-huber}) that the above quotient map induces an isometric embedding in $\h^2_\bb$, so that we obtain the following generalization of Theorem~\ref{thm-main}:
\begin{theorem}
For an amalgamated product $\Gamma=A\ast_C B$ there is a linear isometric embedding
\[
\D(A/\langle\langle C\rangle\rangle)\infplus\D(B/\langle\langle C\rangle\rangle)\hooklongrightarrow\h^2_\bb(\Gamma,\R)
\]
which maps the pair $(f_A,f_B)$ to the bounded cohomology class $\pi^*\omega_f=\omega_{\pi^*f}$ of the pullback of the split quasimorphism $f=f_A\ast f_B$.
\end{theorem}
\begin{example} For the splitting $\mathrm{SL}(2,\Z)=\Z/4\Z\ast_{\Z/2\Z}\Z/6\Z$ the above quotient is equal to $\PSL(2,\Z)=\Z/2\Z\ast\Z/3\Z$. In this case we have a unique split class, see Proposition~\ref{rade}.
\end{example}
A more interesting application concerns surface groups. Let $\Sigma_{n,k}$ be the compact orientable surface of genus $m$ with $k$ boundary components, and let $\Gamma_{m,k}=\pi_1(\Sigma_{m,k})$. These groups (except the abelian ones) belong to both of the classes in Corollary~\ref{cor-raag}, so that we already know that their $\h^2_\bb$ contains a copy of $\D(\Z)\infplus\D(\Z)$. Glueing $\Sigma_{m,1}$ with $\Sigma_{n,1}$ along the boundaries yields the surface $\Sigma_{m+n,0}$. On the level of fundamental groups this corresponds to an amalgamation
\[
\Gamma_{m+n,0}=\Gamma_{m,1}\ast_{\langle\gamma\rangle}\Gamma_{n,1},
\]
where $\gamma$ is a generator for the cyclic fundamental group of the glueing curve. The corresponding free product quotient is
\[
\pi:\Gamma_{m+n,0}\longrightarrow\Gamma_{m,0}\ast\Gamma_{n,0}.
\]
which is induced by the pinching map $\Sigma_{m+n,0}\longrightarrow\Sigma_{m,0}\vee\Sigma_{n,0}$ that contracts the glueing curve to a point. We thus have the following
\begin{theorem}
\label{thm-surface}
Let $\Gamma_m$ be the fundamental group of the closed orientable surface of genus $m$. For $m,n\geq1$ there is a linear isometric embedding
\[
\D(\Gamma_m)\infplus\D(\Gamma_n)\hooklongrightarrow\h^2_\bb(\Gamma_{m+n},\R).
\]
\end{theorem}
One can obtain such embeddings more generally for surfaces with non-empty boundary and for suitable splittings of higher dimensional manifolds along $\pi_1$-injective codimension one submanifolds.
\subsection{Actions of automorphisms}
For a group $\Gamma$ we have the natural action of $\Out(\Gamma)$ on $\h^2_\bb(\Gamma,\R)$ which we discussed in the introduction. For the cohomology class $\omega_f$ of a quasimorphism $f$ this action is given by $\tau.\omega_f=\omega_{\tau.f}$, where we write $\tau.f:=f\circ\tau^{-1}$. It turns out that on split quasimorphisms and classes there is another action by automorphisms. To describe this action we fix a splitting $\Gamma=A\ast B$ and denote by $S\subset\h^2_\bb(\Gamma,\R)$ the corresponding subspace of split classes, that is, the image of the embedding of $\D(A)\oplus\D(B)$. For $\tau\in\Aut(\Gamma)$ we have the induced splitting $\Gamma=\tau(A)\ast\tau(B)$, we denote its space of split classes by $S^\tau$. There is a natural isometric isomorphism $S\longrightarrow S^\tau$ which comes from a map on the level of quasimorphisms, namely
\[
f=f_A\ast f_B\quad\mapsto\quad f^\tau:=f_{\tau(A)}\ast f_{\tau(B)}
\]
where
\[
f_{\tau(A)}=f_A\circ \tau^{-1}|_{\tau(A)},\,\, f_{\tau(B)}=f_B\circ \tau^{-1}|_{\tau(B)},
\]
which induces
\[
S\longrightarrow S^\tau,\quad \omega_f\mapsto\omega_{f^\tau}.
\]
In general the two quasimorphisms $\tau.f$ and $f^\tau$ are not at bounded distance. However, in case of an inner automorphism $\sigma$, where we know that $\sigma.f$ is at bounded distance from $f$, the same holds true for $f^\sigma$, so that all three of $f,\sigma.f,f^\sigma$ are at bounded distance:
\begin{proposition}
\label{prop-conj}
If $\sigma$ is an inner automorphism then the split quasimorphism $f$ is at bounded distance from $f^\sigma$, in particular we have $\omega_f=\omega_{f^\sigma}$ and $S=S^\sigma$. Therefore the space $S^\tau$ only depends on the outer class of an automorphism $\tau$.
\end{proposition}
\begin{proof} Let $\sigma=\mathrm{inn}_h:g\mapsto g^h:=h^{-1}gh$, and let $g=a_1b_1\cdots a_nb_n\in\Gamma$. We have $g=ha_1^{\,h}b_1^{\,h}\cdots a_n^{\,h}b_n^{\,h}h^{-1}$ so $f^\sigma(g)$ is at distance at most $2\cdot\deff f^\sigma$ from
\[
f^\sigma(h)+f^\sigma(a_1^{\,h}b_1^{\,h}\cdots a_n^{\,h}b_n^{\,h})+f^\sigma(h^{-1})=f^\sigma(a_1^{\,h}b_1^{\,h}\cdots a_n^{\,h}b_n^{\,h})=f(a_1a_2\cdots a_nb_n)=f(g).\qedhere
\]
\end{proof}
\begin{proposition}
\label{finite-factors}
If the group $\Gamma$ admits a splitting $\Gamma=A\ast B$ with finite factors $A$,$B$ then the subspace $S\subset\h^2_\bb(\Gamma,\R)$ is independent of the choice of a splitting.
\end{proposition}
\begin{proof} As a consequence of Kurosh's theorem (see, e.g., \cite{Se}, Theorem 14) every splitting of $\Gamma$ is a conjugate $A^g\ast B^g$ for some $g\in\Gamma$, so the claim follows from the previous proposition.
\end{proof}
The minimal example of a group with a non-trivial split quasimorphism is the modular group $\Gamma=\PSL(2,\Z)\cong\Z/2\Z\ast\Z/3\Z$. (Note that the infinite dihedral group $\Z/2\Z\ast\Z/2\Z$ has only trivial quasimorphisms, as it is virtually cyclic.) By the previous proposition the space of split classes $S\subset\h^2_\bb(\Gamma,\R)$ does not depend on the choice of a particular splitting. We have $\D(\Z/2\Z)\oplus\D(\Z/3\Z)\cong\{0\}\oplus\R$, so $S$ is one-dimensional. It is generated by the class of a split quasimorphism
\[
\mathfrak{R}:\PSL(2,\Z)\longrightarrow\Z
\]
which is known as the \emph{Rademacher function}, a function appearing in several different areas of mathematics. A description of $\mathfrak{R}$ as a quasimorphism, from which one can easily see that it splits, was given by Barge-Ghys (\cite{BG}, p.~246). Thus we have 
\begin{proposition}
\label{rade}
For $\Gamma=\PSL(2,\Z)$ there is, up to scaling, a unique non-zero split class $\omega_{\mathfrak{R}}\in\h^2_\bb(\Gamma,\R)$. It is the class associated to the Rademacher function.
\end{proposition}
We may go further and define the total subspace $\s_\Gamma\subset\h^2_\bb(\Gamma,\R)$ of split classes of $\Gamma$ to be
\begin{align*}
\s_\Gamma:&=\mathrm{span}\{\omega\,|\,\mbox{$\omega$ is a split class for some splitting of $\Gamma$}\}\\
&=\mathrm{span}\{\omega_f\,|\,\mbox{$f$ is a split quasimorphism for some splitting of $\Gamma$}\}.
\end{align*}
The group $\Aut(\Gamma)$ acts on $\s_\Gamma$ via the linear extension of the assignment $\omega_f\mapsto\omega_{f^\tau}$, $\tau\in\Aut(\Gamma)$. By Proposition~\ref{prop-conj} these actions descend to $\Out(\Gamma)$.
\begin{example}
\begin{itemize}
\item[(i)] If $\Gamma=A\ast B$ with finite factors then $\s_\Gamma$ is equal to the finite-dimensional space $S$ associated to the given splitting, or any splitting, by Proposition~\ref{finite-factors}. Hence we have a finite-dimensional representation $\s_\Gamma$ of the finite group $\Out(\Gamma)$. Compare this to the usual representation on the infinite dimensional space $\h^2_\bb(\Gamma,\R)$.
\item[(ii)] In $\F_2$ any two splittings are related via an automorphism, so that $\s_{\F_2}=\mathrm{span}\{S^\tau\,|\,\tau\in\Out(\F_2)\}$, where $S$ is the space associated to a preferred splitting.
\end{itemize}
\end{example}
We have no good understanding of the spaces $\s_\Gamma$ in case they have infinite dimension. It would be interesting to know 
\begin{questions}
\begin{itemize}
\item[(i)] How large is the Gromov-norm closure $\overline{\s_{\F_2}}\subset\h^2_\bb(\F_2,\R)$ ?
\item[(ii)] Is the action of $\Out(\Gamma)$ on $\s_\Gamma$ by isometries, so that it extends to an isometric action on the Banach space $\overline{\s_\Gamma}$ ?
\item[(iii)] What are the possible stabilizers $\mathrm{Stab}_{\Out(\F_2)}(\omega)$ of classes $\omega\in\s_{\F_2}$ ?
\end{itemize}
\end{questions}
For the standard action of $\Out(\F_2)$ on $\h^2_\bb(\Gamma,\R)$ we are able to give a partial answer to the third of these questions. Namely we show that there exist split quasimorphisms $f$ on $\F_2=\langle a\rangle\ast\langle b\rangle$ such that $f,\widehat{f}$ and $\omega_f$ have infinite stabilizers. For this purpose we consider the automorphism
\[
\tau_n:\left\{\begin{array}{l}a\mapsto a\\b\mapsto a^nb\end{array}\right.
\]
\begin{theorem}
\label{thm-fixed-point}
Let $f=f_A\ast f_B$ be a split quasimorphism on $\F_2=\langle a\rangle\ast\langle b\rangle$, with bounded factors $f_A$, $f_B$. For each $n\in\Z\backslash\{0\}$ the following are equivalent
\begin{itemize}
\item[(i)] $\tau_n.f=f$
\item[(ii)] $\tau_n.\widehat{f}=\widehat{f}$
\item[(iii)] $\tau_n.\omega_f=\omega_f$
\item[(iv)] The function $f_A$ is $|n|$-periodic and the function $f_B$ is equal to zero.
\end{itemize}
Furthermore, if $|n|\leq2$ these conditions imply that $f=0$.
\end{theorem}
\begin{corollary}
\label{cor-infinite-stabilizer}
If $f_A\in\D(\langle a\rangle)$ is periodic then for $f=f_A\ast0$ the stabilizers $\Stab_{\Aut(\F_2)}(f)$, $\Stab_{\Out(\F_2)}(\widehat{f})$ and $\Stab_{\Out(\F_2)}(\omega_f)$ are infinite.
\end{corollary}
\begin{proof} By the theorem these stabilizers contain the automorphism $\tau_n$ (or its outer class) which has infinite order in $\Aut(\F_2)$ (or in $\Out(\F_2)$).
\end{proof}
\noindent Write $\mathrm{Fix}(\tau)=\{\omega\in\h^2_\bb(\Gamma,\R)\,|\,\tau.\omega=\omega\}$ for the subspace of cohomology classes that are invariant under $\tau\in\Aut(\Gamma)$.
\begin{corollary} For $n\neq0$ the intersection of $\mathrm{Fix}(\tau_n)$ with the space of split classes $S$ is isometrically isomorphic to $\D(\Z/n\Z)$. In particular this intersection is trivial for $n\in\{\pm1,\pm2\}$.
\end{corollary} 
\begin{proof} By Proposition~\ref{prop-ses} the quotient map $\langle a\rangle\longrightarrow\langle a\,|\,a^n=1\rangle\cong\Z/n\Z$ induces an isometric embedding $\D(\langle a\,|\,a^n=1\rangle)\hooklongrightarrow\D(\langle a\rangle)$, the image of which consists precisely of the $n$-periodic functions.
\end{proof}
\begin{proof}[Proof of the theorem] The last statement follows since an alternating $n$-periodic function on $\Z$ is zero when $n\in\{1,2\}$. The implications $\mbox{(i)}\Rightarrow\mbox{(ii)}\Rightarrow\mbox{(iii)}$ are obvious. In order to prove $\mbox{(iii)}\Rightarrow\mbox{(iv)}$ assume that $\tau_n.\omega_f=\omega_f$, equivalently $\tau_{-n}.\omega_f=\omega_f$, which is equivalent to
\[
\label{eq}
\widehat{f}\circ\tau_n=\widehat{f}+\varphi\tag{$\ast$}
\]
for some $\varphi\in\Hom(\F_2,\R)$. Since $\tau_{-n}=\tau_n^{-1}$ we may assume that $n\geq1$. We evaluate \eqref{eq} for different group elements, where we make repeated use of Proposition~\ref{prop-homog}. We write $f_A(k)$ instead of $f_A(a^k)$ and likewise for $B$. Since $\widehat{f}(a)=0$ the equation yields $\varphi(a)=0$.  For $k\neq0$ and $l\geq1$ let $g=a^kb^l\in\F_2$. We have $\tau_n(g)=a^{k+n}b(a^nb)^{l-1}$, so that \eqref{eq} evaluated at $g$ reads
\[
f_A(k+n)+f_B(1)+(l-1)[f_A(n)+f_B(1)]=f_A(k)+f_B(l)+l\varphi(b)
\]
which we rearrange to
\[
f_A(k+n)-f_A(n)+l[f_A(n)+f_B(1)-\varphi(b)]=f_A(k)+f_B(l).
\]
Since the right hand side is bounded as a function of $l$ the bracket vanishes, and we rearrange again to obtain
\[
f_A(k+n)=f_A(k)+[f_A(n)+f_B(l)].
\]
Since $f_A$ is bounded this implies that the bracket in this new equation vanishes, and hence that $f_A(k+n)=f_A(k)$ for all $k\neq 0$. We are left with showing that $f_A(n)=0$ which will imply that $f_A$ is $|n|$-periodic and, since the bracket in the last equation vanishes for all $l\geq1$, that $f_B=0$. To do this we evaluate \eqref{eq} on the commutator $[a,b]$. The right hand side vanishes and we have $\tau_n([a,b])=a^{1+n}ba^{-1}b^{-1}a^{-n}$, so that the evaluation yields
\[
f_A(1+n)-f_A(1)-f_A(n)=0
\]
which implies that $f_A(n)=0$, since $f_A(1+n)=f_A(1)$. We finally prove the implication $\mbox{(iv)}\Rightarrow\mbox{(i)}$. We have to show that if $f_A$ is $n$-periodic and $f_B=0$ then $f=f_A\ast f_B$ is such that for all $g\in\F_2$ we have $f(\tau_n(g))=f(g)$. Consider the quotient map $\pi:\F_2=\langle a\rangle\ast\langle b\rangle\longrightarrow\langle a\,|\,a^n=1\rangle\ast\langle b\rangle$. We have $\pi\circ\tau_n=\tau_n$, and this means that every power $a^{k_i}$ in $g=a^{k_1}b^{l_1}\cdots a^{k_n}b^{l_n}$ corresponds to a power $a^{k_i+p\cdot n}$ in the factorization of $\tau_n(g)$ for some $p\in\Z$, and all other powers of $a$ in $\tau_n(g)$ are of the form $a^{\pm n}$. Since $f_A$ is $n$-periodic we deduce that $f(g)=f_A(k_1)+\dots+f_A(k_n)=f(\tau_n(g))$.
\end{proof}
Note that the automorphisms $\tau_n$ are reducible as they fix the free factor $\langle a\rangle$. We have some evidence that supports the following
\begin{conjecture} If $\tau\in\Out(\F_2)$ is irreducible then $\tau.S\cap S=\{0\}$, in particular the stabilizer $\Stab_{\Out(\F_2)}(\omega_f)$ of every split class $\omega_f$ consists of reducible outer automorphisms.
\end{conjecture}
\subsection{A relation to counting quasimorphisms}
The standard example for a non-trivial quasimorphism on a free group $\F_2=\langle a,b\rangle$ is Brooks' counting quasimorphism, which is defined as follows (see \cite{B}): For $w,g\in\F_2$ we denote by $h_w(g)$ the number of occurences of $w$ as a subword of $g$, when these elements are expressed as reduced words over the given generators. If either of $w,g$ is trivial we set $h_w(g)=0$. Here we allow overlaps, so that for example $h_{aba}(ababa)=2$. The counting quasimorphism associated to the word $w$ is given by $C_w:=h_w-h_{w^{-1}}\in\QMa(\F_2)$. It is non-trivial whenever $w\not\in\{e,a^{\pm1},b^{\pm1}\}$, and in fact, the classes induced by the family $\{C_w\}_{w\in\F_2}$ span an infinite dimensional subspace of $\h^2_\bb(\F_2,\R)$. (There are however non-trivial linear dependencies, see \cite{Gr}, Assertion 5.1) Note that for every coefficient function $\lambda:\F_2\longrightarrow\R$, the infinite linear combination
\[
\sum_{w\in\F_2}\lambda(w)C_w
\]
converges to a map $f:\F_2\longrightarrow\R$, in the topology of pointwise convergence in $\Map(\F_2,\R)$. This is due to the fact that for each $g\in\F_2$, the set $\{w\in\F_2\,|\,C_w(g)\neq0\}$ is finite. The subspace $\QM(\F_2)$ is not closed in $\Map(\F_2,\R)$ and it is not clear when the limit $f$ is itself a quasimorphism. In \cite{Gr} Grigorchuk pointed out that a sufficient, but not necessary condition is that $\lambda$ be an $\ell^1$-function. In the same article he showed that a suitably chosen family of counting quasimorphisms forms a Schauder basis for the space of homogenous quasimorphisms:
\begin{theorem}[\cite{Gr}, Theorem 5.7] There exists a family $W\subset\F_2=\langle a,b\rangle$ such that for every homogenous quasimorphism $f:\F_2\longrightarrow\R$ which vanishes on the generators, there is a unique function $\alpha:W\longrightarrow\R$ with
\[
f=\sum_{w\in W}\alpha(w)\widehat{C}_w.
\]
\end{theorem}
The representation of a class in $\eh^2_\bb(\Gamma,\R)$ by a homogenous quasimorphism $f$ is unique up to homomorphisms, so that it is unique when we require $f$ to vanish on a given generating set. The homogenization of a split quasimorphism $f=f_A\ast f_B$ on $\F_2$, with bounded factors $f_A,f_B$, has the property that $\widehat{f}(a)=\widehat{f}(b)=0$ and has thus an associated coefficient function $\alpha$ from Grigorchuk's theorem. However, the computation of the precise values $\alpha(w)$, which is done recursively in the proof, turns out to be impractical even for the simplest choices for $f_A,f_B$. Our following observation says that split quasimorphisms actually admit a very explicit decomposition into a linear combination of counting quasimorphisms. For $k\geq1$ we use the abbreviations
\begin{align*}
C_{a,k}&:=C_{ba^kb}+C_{ba^kb^{-1}}+C_{b^{-1}a^kb}+C_{b^{-1}a^kb^{-1}}\\
C_{b,k}&:=C_{ab^ka}+C_{ab^ka^{-1}}+C_{a^{-1}b^ka}+C_{a^{-1}b^ka^{-1}}.
\end{align*}
\begin{theorem}Let $f=f_A\ast f_B$ be a split quasimorphism on $\F_2=\langle a\rangle\ast\langle b\rangle$ with $f_A,f_B$ bounded. Then $f$ is at bounded distance from the quasimorphism
\begin{align*}
\sum_{k=1}^\infty f_A(a^k)C_{a,k}+f_B(b^k)C_{b,k},
\end{align*}
in particular, the homogenization can be expressed as
\begin{align*}
\widehat{f}=\sum_{k=1}^\infty f_A(a^k)\widehat{C}_{a,k}+f_B(b^k)\widehat{C}_{b,k},
\end{align*}
and furthermore, if both $f_A$ and $f_B$ have finite support then $f$ is at bounded distance from a finite linear combination of counting quasimorphisms.
\end{theorem}
It is worthwile to note that none of the words $ba^kb$, $ba^kb^{-1}$, etc. are contained in Grigorchuk's set $W$, which consists of words that are of minimal length in their conjugacy class and that don't have a prefix equal to a suffix. 
\begin{proof} Let $F$ be the function given by infinite sum in the theorem. Let, $g=a^{k_1}b^{k_2}\cdots a^{k_{n-1}}b^{k_n}\in\F_2$. The power $b^{k_2}$ is detected by exactly one of the four counting quasimorphisms appearing in $C_{b,k_2}$, depending on the signs of $k_1$ and $k_3$. It is counted with weight $f_B(b^{k_2})$, also if $k_2<0$ as both $C_{b,k_2}$ and $f_B$ are alternating. The same is true for all the powers $b^{k_2},a^{k_3},\dots,b^{k_{n-2}},a^{k_{n-1}}$. On the other hand, the quasimorphisms appearing in $F$ count nothing but these powers, so that
\[
F(g)=f(g)-f_A(a^{k_1})-f_B(b^{k_n}),
\]
which proves that the difference $F-f$ is bounded (and that $F$ is a quasimorphism).
\end{proof}
\section{Split quasi-representations}
Let $G=(G,d)$ be a group endowed with a bi-invariant metric. For a set $X$ we have an induced distance on the set of maps $X\longrightarrow G$ which is given by $d(f_1,f_2)=\sup_{x\in X}d(f_1(x),f_2(x))$. We say that $f_1,f_2$ are at bounded distance if $d(f_1,f_2)<\infty$, and we say that $f$ is bounded if it is at bounded distance from the constant map $x\mapsto e$, in which case we write $\|f\|_\infty$ for this distance. A map $\mu:\Gamma\longrightarrow G$ is called a \emph{quasi-representation} (or $\varepsilon$-representation or $\delta$-homomorphism) if the maps $\Gamma\times\Gamma\longrightarrow G$,
\[
(g,g')\mapsto \mu(gg')\qquad\mbox{and}\qquad (g,g')\mapsto \mu(g)\mu(g')
\]
 are at bounded distance. In this case the distance between these maps is denoted by $\deff\mu$. Note that quasi-representations with values in $G=\R$ are nothing but quasimorphisms. We write $\QRep(\Gamma,G)$ for the set of quasi-representations $\Gamma\longrightarrow G$ and
\[
\QRepa(\Gamma,G)=\left\{\mu\in\QRep(\Gamma,G)\,:\,\mu(g^{-1})=\mu(g)^{-1}\right\}
\]
for the subset of alternating quasi-representations. For every quasi-representation $\mu:\Gamma\longrightarrow G$ we have the associated quantity
\[
D(\mu):=\inf\{d(\mu,\rho)\,|\,\rho\in\Hom(\Gamma,G)\}
\]
which measures the minimal distance to an actual representation.

As a straightforward generalization of split-quasimorphisms we obtain split quasi-representations on $\Gamma=A\ast B$ as follows: For $\mu_A\in\QRepa(A,G)$ and $\mu_B\in\QRepa(B,G)$ we define $\mu=\mu_A\ast\mu_B:\Gamma\longrightarrow G$ by
\begin{align*}
&(\mu_A\ast\mu_B)(a_1b_1\cdots a_nb_n):=\\
&\mu_A(a_1)\mu_B(b_1)\cdots\mu_A(a_n)\mu_B(b_n).
\end{align*}
Due to the bi-invariance of the metric on $G$, the proof of Proposition~\ref{prop-qc} applies in this non-commutative setting as well and we obtain
\begin{proposition} The map $\mu=\mu_A\ast\mu_B$ is an alternating quasi-representation with $\deff \mu=\max\{\deff \mu_A,\deff \mu_B\}$. The induced map
\[
\QRepa(A,G)\times\QRepa(B,G)\longrightarrow\QRepa(\Gamma,G),\quad (\mu_A,\mu_B)\mapsto\mu
\]
extends the natural isomorphism
\[
\Hom(A,G)\times\Hom(B,G)\longrightarrow\Hom(\Gamma,G).
\]
\end{proposition}
In order to make a statement about the quantity $D(\mu)$ for a split quasi-representation $\mu$ we assume that the target group $(G,d)$ has \emph{no $\varepsilon$-small subgroups}, which means that  the open $\varepsilon$-ball around the identity contains no non-trivial subgroup. We obtain the following result
\begin{theorem}
\label{thm-qr}
Let $\Gamma=A\ast B$ and let $G=(G,d)$ be a group without $\varepsilon$-small subgroups. For bounded alternating maps $\mu_A:A\longrightarrow G$, $\mu_B:B\longrightarrow G$ with
\[
\delta:=\max\{\|\mu_A\|_\infty,\|\mu_B\|_\infty\}\leq\frac{\varepsilon}{2}
\]
the split quasi-representation $\mu=\mu_A\ast\mu_B:\Gamma\longrightarrow G$ satisfies
\[
D(\mu)\geq\delta.
\]
\end{theorem}
\begin{proof}For $\delta=0$ the statement is trivial, so we may assume that $\delta>0$ and that there exists $\varphi\in\Hom(\Gamma,G)$ with $d(\mu,\varphi)<\delta$. For all $a\in A$ we have
\[
d(\varphi(a),e)\leq d(\varphi(a),\mu_A(a))+d(\mu_A(a),e)<\delta+\delta\leq\varepsilon
\]
which means that the subgroup $\varphi(A)<G$ is $\varepsilon$-small and hence trivial. The same argument shows that $\varphi(B)$ is trivial, so the homormorphism $\varphi$ is trivial. This means that the map $\mu$ is bounded with $\|\mu\|_\infty<\delta$. Now let $a\in A$ and $b\in B$ be different from the identity and let $g_\pm:=ab^{\pm1}$. By construction we have $\mu(g_\pm)^n=\mu(g_\pm^n)$. This means that the cyclic subgroups $\langle\mu(g_+)\rangle$ and $\langle\mu(g_-)\rangle$ of $G$ are $\delta$-small and hence trivial. In particular we have $\mu(g_\pm)=\mu_A(a)\mu_B(b)^{\pm1}=e$, which implies $\mu_A(a)^2=e$. The subgroup $\{e,\mu_A(a)\}<G$ is again $\delta$-small, so that $\mu_A(a)=e$. It follows that $\mu_A\equiv e$, and likewise $\mu_B\equiv e$. Hence $\delta=0$, a contradiction.
\end{proof}

\appendix
\section{$\h^2_\bb(\F_2,\R)$ is infinite dimensional, a simple proof}
Let $\F_2=\langle a,b\rangle$ be the free group of rank 2. Given a bounded sequence $s:\N\longrightarrow\R$ we define a map $f_s:\F_2\longrightarrow\R$ as follows:

Extend the sequence to an alternating map $s:\Z\longrightarrow\R$, i.e. set $s(0)=0$ and $s(-k)=-s(k)$ for $k<0$. Define $f_s(e)=0$ and for $e\neq g\in\F_2$ with normal form $g=a^{k_1}b^{k_2}\cdots a^{k_{n-1}}b^{k_n}$ let
\[
f_s(g):=s(k_1)+s(k_2)+\ldots+s(k_n).
\]
\begin{theorem_app}The map $f_s$ is a quasimorphism which is trivial if and only if $s=0$. Hence we have a linear embedding
\[
\ell^\infty\hooklongrightarrow\h^2_\bb(\F_2,\R),\qquad s\mapsto\left[\partial f_s\right]_\bb
\]
and the space $\h^2_\bb(\F_2,\R)$ is infinite dimensional.
\end{theorem_app}
\begin{proof}
For $g,h\in\Gamma$ write $\partial f_s(g,h)=f_s(gh)-f_s(g)-f_s(h)$. If (the normal form of) $g$ ends with an $a$-letter and $h$ begins with $b$-letter or vice versa, then $\partial f_s(g,h)=0$ since in this case the normal form of $gh$ is the concatenation of the normal forms of $g$ and $h$. If the normal forms are $g=g'a^k$ and $h=a^{-k}h'$ then $\partial f_s(g,h)=\partial f(g',h')$, since $s(-k)=-s(k)$, and likewise for $b$-letters. So we may assume that $g=g'a^k$ and $h=a^lh'$ with $k+l\neq0$, or likewise with $b$-letters. In this case we have the normal form $gh=g'a^{k+l}h'$ and so $\partial f_s(g,h)=s(k+l)-s(k)-s(l)$, i.e. $|\partial f_s(g,h)\|\leq3\|s\|_\infty$. This proves that $f_s$ is a quasimorhpism.

Now assume that $f_s$ is trivial, i.e. that $f_s=\varphi+\beta$ where $\varphi$ is a homomorphism and $\beta$ is a bounded map. Evaluating this equation at $a^n$ yields $s(n)=n\cdot\varphi(a)+\beta(a^n)$. Since $s$ and $\beta$ are bounded this means that $\varphi(a)=0$, and likewise $\varphi(b)=0$. So $\varphi=0$, which is to say that $f_s=\beta$ is bounded. For $k\in\Z$, $k\neq0$, we have $f_s((a^kb^{\pm1})^n)=n\cdot(s(k)\pm s(1))$. Since $f_s$ is bounded it follows that $s(k)\pm s(1)=0$, so $s(k)=0$ and therefore $s=0$.
\end{proof}
\section{Defect spaces}
\label{appendixB}
Let $\Gamma$ be a group. The \emph{defect space} of $\Gamma$, denoted by $\D(\Gamma)$, is the space of functions $f:\Gamma\longrightarrow\R$ that are bounded and alternating (i.e. $f(g^{-1})=-f(g)$ for all $g\in\Gamma$), equipped with the norm
\[
\dn{f}=\deff{f}=\sup_{g,h\in\Gamma}|\partial f(g,h)|,
\]
where $\partial f(g,h)=f(g)+f(h)-f(gh)$. This is indeed a norm: If $\dn{f}=0$ then $f$ is a bounded homomorphism into $\R$ and hence equal to zero.

In the following statement $\mathrm{ord}(g)$ stands for the (possibly infinite) order of a group element $g\in\Gamma$.
\begin{proposition_app}
\label{def-infty}
For $f\in\D(\Gamma)$ and $g\in\Gamma$, $g\neq1$, we have the estimate
\[
|f(g)|\leq\left(1-\frac{2}{\mathrm{ord}(g)}\right)\dn{f}.
\]
\end{proposition_app}
\begin{proof}
We may assume that $\dn{f}=1$. For $n\geq1$ we have the estimate $|f(g^n)-nf(g)|\leq n-1$, which follows from
\[
|f(g^n)-nf(g)|=\left|\sum_{i=1}^{n-1}f(g)+f(g^i)-f(g^{i+1})\right|\leq(n-1)\dn{f}.
\]
If $g$ has finite even order $2k$ then $f(g^k)=0$, as $f$ is alternating, so the above estimate implies $k|f(g)|\leq k-1$ which means that $|f(g)|\leq1-\frac1k=1-\frac{2}{\mathrm{ord(g)}}$. If $g$ has order $2k+1$ then we have $f(g^k)+f(g^{k+1})=0$, so summation of the estimates
\begin{align*}
|f(g^k)-kf(g)|&\leq k-1\\
|f(g^{k+1})-(k+1)f(g)|&\leq k
\end{align*}
yields $(2k+1)|f(g)|\leq2k-1$, so $|f(g)|\leq1-\frac{2}{2k+1}=1-\frac{2}{\mathrm{ord(g)}}$. Finally, if $g$ has infinite order then letting $n$ tend to infinity in $|\frac{f(g^n)}{n}-f(g)|\leq1-\frac1n$ yields $|f(g)|\leq1$.
\end{proof}
\begin{corollary_app}
\label{norm-equiv}
The defect norm is equivalent to the supremum norm $\|\cdot\|_\infty$, more precisely, for $f\in\D(\Gamma)$ we have
\[
\|f\|_\infty\leq\dn{f}\leq3\|f\|_\infty.
\]
The space $\D(\Gamma)$ is therefore a Banach space, it is non-separable when $\Gamma$ has infinitely many elements of order different from $2$.
\end{corollary_app}
\begin{proof}
The lower bound is a consequence of the proposition, the upper bound is immediate from the definition of the defect norm.
\end{proof}
\begin{proposition_app} An epimorphism $\pi:\Gamma\longrightarrow Q$ induces an isometric embedding $\pi^*:\D(Q)\hooklongrightarrow\D(\Gamma)$, $f\mapsto f\circ\pi$.
\end{proposition_app}
\begin{proof}By surjectivity of $\pi$ we have
\[
\dn{\pi^*f}=\sup_{g,h\in\Gamma}|\partial f(\pi(g),\pi(h))|=\sup_{g,h\in Q}|\partial f(g,h)|=\dn{f}.\qedhere
\]
\end{proof}
\begin{proposition_app}
\label{subgroup-emb}
For a monomorphism $i:H\longrightarrow\Gamma$, the map
\[
s_i:\D(H)\longrightarrow\D(\Gamma),\quad s_i(f)(g)=\left\{\begin{array}{cl}f(h),&g=i(h)\\0,&g\not\in i(H)\end{array}\right.
\]
is an isometric embedding.
\end{proposition_app}
\begin{proof}
Write $F:=s_i(f)$. We identify $H$ with its image $i(H)$. Let $g,h\in\Gamma$. Of the three elements $g,h,gh$ either none, one or all three belong to $H$. In the first case we have $\partial F(g,h)=0$, in the second case we have $|\partial F(g,h)|\leq\|f\|_\infty\leq\dn{f}$ (by Proposition \ref{def-infty}) and in the last case we have $|\partial F(g,h)|=|\partial f(g,h)|\leq\dn{f}$. So we have $\dn{F}\leq\dn{f}$. As $i$ is injective we also have the reverse inequality, i.e. $\dn{F}=\dn{f}$.
\end{proof}
In case of a normal subgroup the last statement can be improved. Consider a short exact sequence
\[
\xymatrix{
1\ar[r]&N\ar[r]^{i}&\Gamma\ar[r]^{\pi}&Q\ar[r]&1.
}
\]
We have the following maps between the defect spaces
\[
\xymatrix{
\D(Q)\ar[r]^{\pi^*}&\D(\Gamma)\ar@<+3pt>[r]^{i^*}&\D(N)\ar@<+1pt>[l]^{s_i},
}
\]
where $i^*\circ\pi^*=0$ and $s_i$ is a section of $i^*$. This means that the embeddings $\pi^*$ and $s_i$ are complementar. More precisely, we have
\begin{proposition_app}
\label{prop-ses}
For a short exact sequence as above we have an isometric embedding
\[
j:\D(N)\infplus\D(Q)\hooklongrightarrow\D(\Gamma)
\]
which is given by $j=s_i+\pi^*$, and, more explicitly, by
\[
j(f,f')(g)=\left\{\begin{array}{cl}f(n),&g=i(n)\\f'(\pi(g)),&g\not\in i(N).\end{array}\right.
\]
\end{proposition_app}
\noindent Here the notation $\infplus$ stands for the max-norm on the direct sum.
\begin{proof}
For $(f,f')\in\D(N)\oplus\D(Q)$ we write $F:=j(f,f')$. We identify $N$ with its image $i(N)$. We first note that for $g\in N\subset\Gamma$ we have  $\pi^*(f')(g)=f'(1_N)=0$, so $F(g)=f(g)$, which proves the explicit formula for $j$. Now let $g,h\in\Gamma$. If none of $g,h,gh$ is contained in $N$ then $|\partial F(g,h)|=|\partial f'(\pi(g),\pi(h))|\leq\dn{f'}$. If $g\in N,h\not\in N$ then $\pi(h)=\pi(gh)$, so
\[
|\partial F(g,h)|=|f(g)+f'(\pi(h))-f'(\pi(gh))|=|f(g)|\leq\|f\|_\infty\leq\dn{f},
\]
by Proposition \ref{def-infty}. Because of the identities
\[
|\partial F(g,h)|=|\partial F(h^{-1},g^{-1})|=|\partial F(g^{-1},gh)|
\]
the same holds true whenever exactly one of $g,h,gh$ is contained in $N$. If all three elements are in $N$ then $|\partial F(g,h)|=|\partial f(g,h)|\leq\dn{f}$. It follows that $\dn{F}\leq\max\{\dn{f},\dn{f'}\}$, and the reverse inequality holds since the maps $s_i$ and $\pi^*$ are isometric.
\end{proof}

\bibliographystyle{abbrv}
\bibliography{bib}

\end{document}